\documentclass[11pt, a4paper]{amsart}
\usepackage{amsfonts}
\usepackage{amssymb}
\usepackage{amsmath, mathrsfs}
\usepackage{pdfsync, leftidx}
\usepackage{amsthm}

\newtheorem{theorem}{Theorem}[section]
\newtheorem{lemma}{Lemma}[section]

\newtheorem{corollary}{Corollary}[section]
\newtheorem{proposition}{Proposition}[section]
\newtheorem{remark}{Remark}

\newenvironment{pf-main}{{\sc Proof of Theorem \ref{mainresult}.}\hspace{3mm}}{\qed}
\newcommand{\nc}{\newcommand}
\nc{\cadlag}{c\`{a}dl\`{a}g } \nc{\caglad}{c\`{a}gl\`{a}d }
\nc{\ba}{\begin{array}} \nc{\ea}{\end{array}}
\nc{\be}{\begin{equation}} \nc{\ee}{\end{equation}}
\nc{\bea}{\begin{eqnarray}} \nc{\eea}{\end{eqnarray}}
\nc{\bean}{\begin{eqnarray*}} \nc{\eean}{\end{eqnarray*}}
\nc{\bu}{\bullet} \nc{\nn}{\nonumber} \nc{\cA}{{\mathcal A}}
\nc{\cB}{{\mathcal B}} \nc{\cC}{{\mathcal C}} \nc{\cD}{{\mathcal
D}} \nc{\cL}{{\mathcal L}} \nc{\cN}{{\mathcal
N}}\nc{\bbD}{\mathbb{D}} \nc{\cG}{{\mathcal G}} \nc{\cI}{{\mathcal I}}\nc{\cF}{{\mathcal
F}} \nc{\cS}{{\mathcal S}} \nc{\cO}{{\mathcal O}}\nc{\cR}{{\mathcal
R}}\nc{\cU}{{\mathcal U}} \nc{\cH}{{\mathcal H}}
\nc{\cK}{{\mathcal K}} \nc{\cM}{{\mathcal M}} \nc{\cP}{{\mathcal
P}} \nc{\bbE}{\mathbb{E}} \nc{\bbEQ}{\mathbb{E}^{\mathbb{Q}}}
\nc{\eps}{\varepsilon}\nc{\bbU}{\mathbb{U}}
\nc{\fps}{(\Om,\cF, (\cF_t)_{t \geq 0}, \bbP)}
\nc{\fpsT}{(\Om,\cF, (\cF_t)_{t \in [0,T]}, \bbP)}
\nc{\bbEP}{\mathbb{E}_{\mathbb{P}}}\nc{\bbL}{\mathbb{L}}
\nc{\bbP}{\mathbb{P}} \nc{\bbQ}{\mathbb{Q}} \nc{\Om}{\Omega}
\nc{\om}{\omega} \nc{\bbR}{\mathbb{R}} \nc{\bbC}{\mathbb{C}}
\nc{\bfr}{\begin{flushright}} \nc{\efr}{\end{flushright}}
\nc{\dXt}{\Delta X_{t}} \nc{\dXs}{\Delta X_{s}}
\nc{\bs}{\blacksquare} \nc{\dX}{\Delta X} \nc{\dY}{\Delta Y}
\nc{\dnkx}{\left(X(T^{n}_{k})-X(T^{n}_{k-1})\right)}
\nc{\dom}{depth-of-the-market } \nc{\uar}{\uparrow}
\nc{\dar}{\downarrow}\nc{\rar}{\rightarrow}
\nc{\half}{\frac{1}{2}}
\nc{\sgn}{\mbox{sgn}}
\nc{\ssgn}{\widetilde{\sgn}}
 \nc{\hbE}{\hat{\bbE}}

\nc{\what}{\widehat} \nc{\fhat}{\what{f}}
 \nc {\parx}{\frac{\partial}{\partial x}} \nc
{\parw}{\frac{\partial}{\partial w}} \nc
{\parww}{\frac{\partial^2}{\partial w^2}}
\def\rar{\rightarrow}
\def\dar{\downarrow}

\nc{\chf}{\mbox{$\mathbf1$}}
\numberwithin{equation}{section}
\nc{\eid}{\stackrel{d}{=}}
\begin{document}

\title{Filtered Az\'ema martingales}
\date{\today}

\author[]{Umut \c{C}etin}
\address{Department of Statistics, London School of Economics and Political Science, 10 Houghton st, London, WC2A 2AE, UK}
\email{u.cetin@lse.ac.uk}
\begin{abstract} We study the optional projection of a standard Brownian motion on the natural filtration of certain kinds of observation processes. The observation process, $Y$, is defined as a solution of a stochastic differential equation such that it reveals some (possibly noisy) information about the signs of the Brownian motion when  $Y$ hits $0$. As such, the associated optional projections are related to Az\'ema's martingales which are obtained by projecting the Brownian motion onto the filtration generated by observing its signs. 
\end{abstract}

\keywords{Az\'ema's martingale, excursions of Brownian motion, skew Brownian motion, optional projection, local times.}

\maketitle

\section{Introduction}
Let $\fps$ be a filtered probability space satisfying the usual
conditions and $W$ be a standard Brownian motion with $W_0=0$ and adapted to
$(\cF_t)_{t \geq 0}$. Define $\cG^0_t:=\sigma(\mbox{sgn}(W_s); s\leq
t)$, where
\[
\mbox{sgn}(x)=\left\{\ba{ll}
1, & \mbox{if } x >0;\\
-1, &\mbox{if } x \leq 0,\ea\right .
\]
and let $(\cG_t)_{t \geq 0}$  be the augmentation of $\cG^0_t$ with the $\bbP$-null sets.  {\em Az\'ema's martingale} is obtained by projecting
$W$ onto $\cG$. We will denote the $(\cG, \bbP)$-optional projection
of $B$ with $\mu$.  This martingale first appeared in \cite{A} and was
further studied in  a series of papers such as \cite{AY}, \cite{E} and
\cite{PAM}. Our presentation follows \cite{Pro}.

 By construction Az\'ema's martingale is closely related to the excursions of
 Brownian motion away from $0$. In fact, if we set
\be \label{e:gamma}
\gamma_t:=\sup\{s \leq t: W_s=0\},
\ee
then (see, e.g. \cite{Pro})
\be \label{e:am}
\mu_t=\bbE[W_t|\cG_t]=\sgn(W_t)\frac{\pi}{2}\sqrt{t-\gamma_t}.
\ee
Thus, Az\'ema's martingale is the best estimate, in a mean-square
sense, for the value of a Brownian motion when one only observes its
zeroes and the signs of its excursions.

The above interpretation of $\mu$
was used by \cite{cjpy} to model the default probabilities of a firm
under incomplete information. Assuming cash balances follow a Brownian
motion, \cite{cjpy} defines the default time for the firm as the first
time that its cash balances have remained negative for a certain
amount of time and doubled in absolute value. On the other hand, the
market's only information regarding the cash balances is whether the
firm is in financial distress, i.e. the cash balance is negative, or
not. This information set thus corresponds to $\cG$ in above
notation. Using certain properties of Az\'ema's martingale and some
results from excursion theory the authors explicitly compute the
$\cG$-predictable compensator of the default indicator process. The use of Az\'ema's martingale in Mathematical Finance Theory is not limited to default risk. It is also the key process in models for Parisian barrier options (see \cite{cjy}).

Motivation of this paper comes from the following question: What
happens to the optional projection of Brownian motion when we observe
its signs, possibly with some noise,  at the zeroes of another process which we can observe
continuously? Clearly, the answer to this question depends on how one
defines the observation process. The most common approach in applications is to model the observation process  as  a solution of a stochastic differential equation. In this paper we will look at two
different types of stochastic differential equations for the observation process.

The first formulation that we will consider corresponds to the case
when one imperfectly observes the signs of Brownian motion at the zeroes of an
observation process. Here imperfection corresponds to the case when
the true signal is contaminated with some noise. In view of the
standard nonlinear filtering theory one can model the observation
process as a (weak) solution to the following stochastic
differential equation (SDE):
\be \label{e:fam1}
Y_t=B_t+\alpha \int_0^t \sgn(W_{g_s(Y)})\,ds
\ee
where $\alpha \in \bbR$ and
\be
g_t(Y):=\sup\{s \leq t: Y_s=0\}.
\ee
In Section \ref{s:ffam} we study the existence and uniqueness of (weak) solutions of (\ref{e:fam1}) and the projection of $W$ onto the natural filtration of the solution. The methods employed are standard techniques from nonlinear filtering theory. On the other hand, the existence of a strong solution to (\ref{e:fam1}) remains as an interesting open problem.

Another possibility for modeling the observation process is to
introduce the knowledge on the sign of $W$ through the local times of
$Y$ whose support is contained in the zero set of $Y$. In this case
the corresponding SDE is the following:
\be  \label{e:fam2}
Y_t=B_t +\alpha \int_0^t \sgn(W_s)dL_s,
\ee
where $L$ is the {\em symmetric} local time of $Y$ at $0$. We will see in Section \ref{s:sfam} that the solution to the above equation is closely related to the {\em skew Brownian motion} which we recall next.
\begin{theorem} \label{t:HS} {\em (Harrison and Shepp \cite{HS})} There is a unique
  strong solution, called {\em skew Brownian motion}, to
\be \label{e:sbm}
X_t=B_t + \alpha L_t(X),
\ee
where $L(X)$ is the symmetric local time of $X$ at the level $0$ if and only if $|\alpha|\leq 1$.
\end{theorem}
First appearances of skew Brownian motion in the literature goes back
to as early as \cite{im} and \cite{W}. Formally it is obtained by
changing the sign of a Brownian motion in every excursion depending on
the value of an independent Bernoulli random variable.  A related SDE
introduced by Sophie Weinryb is
\[
X_t=B_t + \int_0^t \alpha(s) dL_s(X),
\]
whose pathwise uniqueness is established in \cite{Weinryb} when
$\alpha$ is a deterministic function taking values in $[-1,1]$.

The reader is referred to the recent survey in \cite{ssbm} where one can find a discussion of different constructions of skew Brownian motion and its properties. In Section \ref{s:sfam} we will prove that there exists a unique {\em strong} solution to (\ref{e:fam2}) and see how it is connected to the solutions of (\ref{e:sbm}). This connection will be helpful in the characterisaton of the  natural filtration of the solution of (\ref{e:fam2}) and the associated projection of $W$, which is our main concern. We will see that this projection changes only by jumps which may only occur at the end of an excursion interval of a skew Brownian motion.
\section{Filtered  Az\'ema martingale of the first kind} \label{s:ffam}
Observe that the drift coefficient of the SDE in (\ref{e:fam1}) is path
dependent and, thus, the classical results on the existence and
uniqueness of strong solutions of SDEs do not apply. However, since
$\sgn$ function is bounded, one can easily construct a weak solution
to this equation on any interval $[0,T]$. Indeed, if $\beta$ and $W$
are two independent Brownian motions in some probability space, one can
define a change of measure via the martingale
\[
\exp\left(\alpha \int_0^t\sgn(W_{g_s(\beta)})d\beta_s -\half \alpha^2 t\right)
\]
and under the new measure $\beta$ solves  (\ref{e:fam1}) while $W$
stays a Brownian motion. The same
Girsanov transform also implies that the law of any weak solution $(W,
Y)$ of  (\ref{e:fam1}) is the same. Let $\cF^Y$ be  the smallest filtration satisfying the
usual conditions and containing the filtration generated by $Y$.
 In the remainder of this section
we will fix a weak solution to  (\ref{e:fam1}) and compute the
corresponding conditional probabilities for this pair. However, the
weak uniqueness of the solutions imply that the conditional laws of $W$
on $\cF^Y$ computed in this section\footnote{One should be careful in
  computing the conditional laws of random variables measurable with
  respect to $\cF_{\infty}$ since the martingale used for the change
  of measure is not uniformly integrable.} do not depend
on the choice of the weak solution.

In the computations performed in this and the subsequent section we
will often make use of the  {\em balayage formula} as given in the
next lemma.
\begin{lemma} \label{l:balayage} {\em (Theorem VI.4.2 in \cite{RY})} If $K$
  is a locally bounded $\cF$-predictable process, $(K_{g_t(Y)}Y_t)_{t
    \geq 0}$ is a continuous semimartingale and satisfies
\[
K_{g_t(Y)}Y_t=\int_0^t K_{g_s(Y)}dY_s.
\]
\end{lemma}
As a first application of the balayage formula, we will now see that
$\sgn(W_{g(B^{(\alpha)})})B^{(\alpha)}$ is a weak solution of
(\ref{e:fam1})  where $B^{(\alpha)}$ is defined by $B^{(\alpha)}
_t=B_t +\alpha t$. Indeed, if we set
$Y_t=\sgn(W_{g_t(B^{(\alpha)})})B_t^{(\alpha)}$, then balayage formula
implies
\[
dY_t=\sgn(W_{g_t(B^{(\alpha)})})dB_t +\alpha\,
\sgn(W_{g_t(B^{(\alpha)})}) dt.
\]
Moreover, $\int_0^{\cdot}\sgn(W_{g_t(B^{(\alpha)})})dB_t$ is a
standard Brownian motion. The claim follows since by construction
$g(Y)=g(B^{(\alpha)})$. Thus, by the uniqueness of weak solutions, we obtain
\be \label{e:Yeid1}
Y\eid \sgn(W_{g(B^{(\alpha)})})B^{(\alpha)}.
\ee
In other words, $Y$ is obtained by changing the sign of a Brownian motion with drift via the sign of an independent Brownian motion sampled at the beginning of the current excursion (away from $0$) of the drifting Brownian motion. As such, the resulting process in a sense is in the same spirit of a skew Brownian motion described in (\ref{e:sbm}), which will be relevant to the filtered Az\'ema martingale of the second kind discussed in the next section.

An immediate consequence of the aforementioned equality in law is the following
\begin{proposition} Let $(Y,W)$ be the unique weak solution of
  (\ref{e:fam1}). Then,
\begin{itemize}
\item[i)] $\lim_{t \rar \infty} |Y_t|=\infty$ and
  $\bbP(Y_{\infty}=\infty)=\bbP(Y_{\infty}=-\infty)=\half$.
\item[ii)] $\bbP(\sup\{t:Y_t=0\}<\infty)=1$.
\end{itemize}
\end{proposition}
\begin{proof} i) follows from the fact that $|B^{(\alpha)}_t| \rar \infty$
  as $t \rar \infty$ and that $W$ is independent of
  $B^{(\alpha)}$. Similarly, since $B^{(\alpha)}$ transient, there is a
last time that it hits $0$. Since the zeroes of $Y$ are the same as
those of $B^{(\alpha)}$, the result follows.
\end{proof}
The above result is another manifestation of that the law of $Y$ is
equivalent to  the law of
a Brownian motion {\em only if} they are stopped at a finite
stopping time. Indeed, if the law of $Y$ were equivalent to the Wiener measure, the zero set of $Y$ would be unbounded with
probability $1$. This discrepancy also confirms that the martingale
used to obtain the measure change is not uniformly integrable.
\begin{remark} If we set $Z_t=\sgn(W_{g_t(Y)})Y_t$ and thereby note
  that $g(Z)=g(Y)$, we obtain via balayage formula
\be \label{e:fam11}
Z_t=\int_0^t \sgn(W_{g_s(Z)})dB_s + \alpha t.
\ee
Let's consider the analogous SDE without drift, i.e.
\be \label{e:sdend}
Z_t=\int_0^t \sgn(W_{g_s(Z)})dB_s.
\ee
Then, there is a unique strong solution to this
equation. Indeed, in view of the balayage formula,
$\sgn(W_{g_t(Z)})Z_t=B_t$. Thus, the zeroes of $Z$ are the zeroes of
$B$ and we have $Z_t=\sgn(W_{g_t(B)})B_t$.

On the other hand, similar arguments do not seem to work for
(\ref{e:fam11}). It is an open question whether this equation admits a strong solution.
\end{remark}

We next obtain the semimartingale decomposition of $Y$ with respect to
its own filtration.
\begin{proposition} \label{p:innovation} Let $(Y,W)$ be the unique weak solution of
  (\ref{e:fam1}). Then,
\begin{itemize}
\item[i)] $\bbE[\sgn(W_{g_t(Y)})|\cF^Y_t]=\tanh(\alpha Y_t)$;
\item[ii)] $Y$ has the following decomposition in its own filtration:
\[
Y_t=B^Y_t+\alpha \int_0^t \tanh(\alpha Y_s)\,ds,
\]
\end{itemize}
where $B^Y$ is an $\cF^Y$-Brownian motion.
\end{proposition}
\begin{proof} Note that ii) follows immediately from i) in view of the
  standard results on filtering, see, e.g. Theorem 8.1 in
  \cite{ls}. To see why i) holds take a constant $T>t$ and consider the measure $\bbQ_T\sim \bbP_T$ under
  which $(Y_s)_{s \in [0,T]}$ is a Brownian motion independent of
  $(W_s) _{s \in [0,T]}$ where $\bbP_T$ is the restriction of $\bbP$
  to $\cF_T$. Then, it follows from  Girsanov's theorem that
\bean
\bbE[\sgn(W_{g_t(Y)})|\cF^Y_t]&=&\frac{\bbE^{\bbQ}\left[\sgn(W_{g_t(Y)})
    \exp\left(\alpha \int_0^t\sgn(W_{g_s(Y)})dY_s -\half \alpha^2
      t\right)\bigg|\cF^Y_t\right]}{\bbE^{\bbQ}\left[\exp\left(\alpha
      \int_0^t\sgn(W_{g_s(Y)})dY_s -\half \alpha^2
      t\right)\bigg|\cF^Y_t\right]}\\
&=&\frac{\bbE^{\bbQ}\left[\sgn(W_{g_t(Y)})
    \exp\left(\alpha\,
      \sgn(W_{g_t(Y)})Y_t\right)\big|\cF^Y_t\right]}{\bbE^{\bbQ}\left[\exp\left(\alpha\,
      \sgn(W_{g_t(Y)})Y_t\right)\big|\cF^Y_t\right]}\\
&=&\frac{\sinh(\alpha Y_t)}{\cosh(\alpha Y_t)},
\eean
where the second equality follows from Lemma \ref{l:balayage} and the
last equality is due to the independence of $W$ and $Y$ under $\bbQ$
along with the facts that $g_t$ is $\cF^Y_t$-measurable and the
probability that $W_s>0$ is $1/2$ for any $s$.
\end{proof}
Using the same technique as in the proof of the above proposition, we
can obtain the conditional law of $W$.
\begin{theorem} Let $p(t, y-x)$ be the transition density of a standard
  Brownian motion and set
\be \label{e:phi}
\Phi(x):=\int_{-\infty}^x p(1,y)\,dy.
\ee
\begin{itemize}
\item[i)] $\cF^Y_t$-conditional law of $W_t$ has a density, which is given
by
\[
\bbP(W_t\in dx |\cF^Y_t]=p(t,x)\frac{\Phi\left(\sqrt{\frac{g_t}{t(t-g_t)}}x\right)e^{\alpha
    Y_t}+\Phi\left(-\sqrt{\frac{g_t}{t(t-g_t)}}x\right)e^{-\alpha
    Y_t}}{\cosh(\alpha Y_t)}\,dx.
\]
\item[ii)]
Conditional moments of $W$ are given by
\[
\bbE[W^n_t|\cF^Y_t]=\left\{\ba{ll}
  \frac{(2k)!}{\sqrt{\pi} k!} \left(\frac{g_t(Y)}{2}\right)^k, & \mbox{if } n=2k,\\
\frac{k !}{\sqrt{\pi}}\left(2g_t(Y)\right)^{k+\half}\tanh(\alpha Y_t), &
\mbox{if } n=2k +1.
\ea \right .
\]
In particular,
\[
\bbE[W_t|\cF^Y_t]=\sqrt{\frac{2g_t(Y)}{\pi}}\tanh(\alpha Y_t).
\]
\end{itemize}
\end{theorem}
\begin{proof} Let $f:\bbR\mapsto \bbR$ be a bounded measurable
  function. Then,
\[
\bbE[f(W_t)|\cF^Y_t]=\frac{\bbE^{\bbQ}\left[f(W_t)
    \exp\left(\alpha\,
      \sgn(W_{g_t(Y)})Y_t\right)\big|\cF^Y_t\right]}{\cosh(\alpha
    Y_t)}
\]
where $\bbQ$ is the measure defined in the proof of Proposition \ref{p:innovation}.
Moreover, the numerator in the above fraction equals
\be \label{e:num1}
\int_{-\infty}^{\infty}dx f(x) p(t,x) \bbE^{\bbQ}\left[\exp\left(\alpha\,
      \sgn(W_{g_t(Y)})Y_t\right)\big|\cF^Y_t, W_t=x\right]
\ee
due to the independence of $W$ and $Y$ under $\bbQ$. On the other
hand, for any $s \leq t$ the distribution of $W_s$ conditional on
$W_t=x$ is Gaussian with mean $\frac{s}{t}x$ and variance
$\frac{s(t-s)}{t}$. Thus,
\[
\bbP(W_s >0|W_t=x)=
\bbP(\sqrt{\frac{s(t-s)}{t}}W_1+\frac{s}{t}x)=\Phi\left(\sqrt{\frac{s}{t
    (t-s)}} x\right).
\]
Utilising once more the independence of $Y$
and $W$, we see that (\ref{e:num1}) equals
\[
\int_{-\infty}^{\infty}dx f(x) p(t,x) \left\{\Phi\left(\sqrt{\frac{g_t}{t(t-g_t)}}x\right)e^{\alpha
    Y_t}+\Phi\left(-\sqrt{\frac{g_t}{t(t-g_t)}}x\right)e^{-\alpha
    Y_t}\right\}.
\]
This completes the proof of the density.

The conditional moments can be calculated by integrating this density,
which is a lengthy task. However, since for any $\lambda \in \bbR$
$\exp(\lambda W_t-\half t)$ is a martingale
independent of $Y$, and in particular of $g_t(Y)$, one has
\bean
u(\lambda)&:=&\bbE^{\bbQ}\left[\exp(\lambda W_t)
    \exp\left(\alpha\,
      \sgn(W_{g_t(Y)})Y_t\right)\big|\cF^Y_t\right]\\
&=&\bbE^{\bbQ}\left[\exp\left(\lambda
    W_{g_t(Y)}+ \half \lambda^2 (t- g_t(Y))\right)
    \exp\left(\alpha\,
      \sgn(W_{g_t(Y)})Y_t\right)\bigg|\cF^Y_t\right]\\
&=&\exp\left(\half \lambda^2 (t- g_t(Y))\right)\left\{ e^{\alpha Y_t}
  \int_0^{\infty}e^{\lambda x} p(g_t(Y),x)\,dx +e^{-\alpha Y_t}
\int_{-\infty}^0e^{\lambda x} p(g_t(Y),x)\,dx\right\}.
\eean
 Since we can differentiate with respect to $\lambda$ under the
 integral sign, we have
\[
\frac{d^n u}{d\lambda^n}\bigg|_{\lambda=0}= e^{\alpha Y_t}
  \int_0^{\infty}x^n p(g_t(Y),x)\,dx +e^{-\alpha Y_t}
\int_{-\infty}^0x^n p(g_t(Y),x)\,dx.
\]
Moreover, one has
\[
\int_0^{\infty}x^n \frac{1}{\sqrt{2 \pi a}} e^{-\frac{x^2}{2a}}\,dx=\frac{(2a)^{n/2}}{2 \sqrt{\pi}}\Gamma(\frac{n+1}{2})=\left\{\ba{ll}
\frac{(2k)!}{\sqrt{\pi}k! 2^{k+1}} (a)^k, & \mbox{if } n=2k,\\
\frac{k !}{\sqrt{2 \pi}}2^k a^{k+\half}, & \mbox{if } n=2k +1.
\ea \right .
\]
Thus, due to the symmetry of $p$, we obtain
\[
\frac{d^n u}{d\lambda^n}\bigg|_{\lambda=0}= \left\{\ba{ll} 2
  \cosh(\alpha Y_t) \frac{(2k)!}{\sqrt{\pi} k! 2^{k+1}} ( g_t(Y))^k, & \mbox{if } n=2k,\\
2 \sinh(\alpha Y_t)\frac{k !}{\sqrt{2 \pi}}2^k g_t(Y)^{k+\half}, &
\mbox{if } n=2k +1.
\ea \right .
\]
\end{proof}
In view of the above theorem we may define the {\em filtered Az\'ema
  martingale of the first kind} by $\hat{\mu}_t=\sqrt{\frac{2g_t(Y)}{\pi}}\tanh(\alpha
Y_t)$. Observe that, since $\tanh(0)=0$ and $g_t(Y)$ changes value
only when $Y$ hits $0$, $\hat{\mu}$ is a continuous martingale in
contrast to the discontinuous Az\'ema
  martingale, $\mu$.

Although the Brownian motion $W$ is clearly not independent of $Y$, observing $Y$ does not tell us anything new regarding the process $(\gamma_t)$. We will only prove $\gamma_1$ is independent of $Y$. The analogous statement can be proven for any $\gamma_t$ along the same lines.
\begin{proposition} $\gamma_1$ is independent of $\cF^Y$.
\end{proposition}
\begin{proof} Let $t \leq 1$ and consider
\[
\bbE^{\bbQ}\left[f(\gamma_1)\exp\left(\alpha\sgn(W_{g_t(Y)})Y_t\right)\big|\cF^Y_t\right]
\]
for some bounded measurable real function $f$. Observe that
\[
\chf_{[g_t(Y) < \gamma_1]}\bbE^{\bbQ}\left[\exp\left(\alpha\sgn(W_{g_t(Y)})Y_t\right)\big|\cF^Y_t,\gamma_1\right]=\chf_{[g_t(Y) < \gamma_1]}\cosh(\alpha Y_t)
\]
since conditional on $\gamma_1$, $(W_t)_{t \in [0,\gamma_1]}$ is a Brownian bridge (see Exercise XII.3.8 in \cite{RY}) and therefore $\bbQ(W_{g_t(Y)}>0|g_t(Y), \gamma_1)=\half$ on the set $[g_t(Y) < \gamma_1]$.  Moreover,
\[
\chf_{[g_t(Y) > \gamma_1]}\bbE^{\bbQ}\left[\exp\left(\alpha\sgn(W_{g_t(Y)})Y_t\right)\big|\cF^Y_t,\gamma_1\right]=\chf_{[g_t(Y) > \gamma_1]}\cosh(\alpha Y_t),
\]
as well since $g_t(Y)\leq 1$ and therefore $\sgn(W_{g_t(Y)})$ is independent of $\gamma_1$ (see, e.g. Lemme 1 in \cite{AY}). Since $[g_t(Y) = \gamma_1]$ is a $\bbQ$-null set due to their independence and the continuity of the distribution of $\gamma_1$, we deduce that
\[
\bbE^{\bbQ}\left[f(\gamma_1)\exp\left(\alpha\sgn(W_{g_t(Y)})Y_t\right)\big|\cF^Y_t\right]=\cosh(\alpha Y_t)\bbE^{\bbQ}[f(\gamma_1)]=\cosh(\alpha Y_t)\bbE[f(\gamma_1)],
\]
which in turn implies $\bbE[f(\gamma_1)|\cF^Y_t]=\bbE[f(\gamma_1)]$ for any $f$.

To show the independence for $t>1$, note that it suffices to consider
\[
\chf_{[g_t(Y)>1]}\bbE^{\bbQ}\left[f(\gamma_1)\exp\left(\alpha\sgn(W_{g_t(Y)})Y_t\right)\big|\cF^Y_t\right]
\]
since when $\chf_{[g_t(Y)<1]}$ the problem is reduced to the previous case. Notice by the Markov property of $W$ that, given $W_1$, $\gamma_1$ and $\sgn(W_u)$ are independent for any $u>1$. Thus, on $[g_t(Y)>1]$
\bean
\bbE^{\bbQ}\left[f(\gamma_1)\exp\left(\alpha\sgn(W_{g_t(Y)})Y_t\right)\big|\cF^Y_t, W_1\right]&=&\bbE^{\bbQ}[f(\gamma_1)|W_1]\bbE^{\bbQ}\left[\exp\left(\alpha\sgn(W_{g_t(Y)})Y_t\right)\big|\cF^Y_t, W_1\right]\\
&=&\bbE^{\bbQ}[f(\gamma_1)|W_1]\exp(\alpha Y_t)\Phi\left(\frac{W_1}{\sqrt{g_t(Y)-1}}\right)\\
&&+\bbE^{\bbQ}[f(\gamma_1)|W_1]\exp(-\alpha Y_t)\Phi\left(-\frac{W_1}{\sqrt{g_t(Y)-1}}\right),
\eean
where $\Phi$ is the function defined in (\ref{e:phi}). Therefore, on $[g_t(Y)>1]$
\bean
\bbE^{\bbQ}\left[f(\gamma_1)\exp\left(\alpha\sgn(W_{g_t(Y)})Y_t\right)\big|\cF^Y_t\right]&=&\bbE^{\bbQ}\left[f(\gamma_1) \exp(\alpha Y_t)\Phi\left(\frac{W_1}{\sqrt{g_t(Y)-1}}\right)\bigg|\cF^Y_t\right]\\
 &&+\bbE^{\bbQ}\left[f(\gamma_1)\exp(-\alpha Y_t)\Phi\left(-\frac{W_1}{\sqrt{g_t(Y)-1}}\right)\bigg|\cF^Y_t\right].
\eean
On the other hand, the conditional law of $W_1$ given $\gamma_1=s$ is (see Exercise XII.3.8 in \cite{RY})
\[
\frac{|x|}{2(1-s)}\exp\left(-\frac{x^2}{2(1-s)}\right)\,dx.
\]
Using this density, one can directly show that
\[
\bbE^{\bbQ}\left[\Phi\left(\frac{W_1}{\sqrt{g_t(Y)-1}}\right)\bigg|\gamma_1,g_t(Y)\right] =\bbE^{\bbQ}\left[\Phi\left(-\frac{W_1}{\sqrt{g_t(Y)-1}}\right)\bigg|\gamma_1,g_t(Y)\right]= \half.
\]
Hence, we arrive at
\[
\chf_{[g_t(Y)>1]}\bbE^{\bbQ}\left[f(\gamma_1)\exp\left(\alpha\sgn(W_{g_t(Y)})Y_t\right)\big|\cF^Y_t\right] =\chf_{[g_t(Y)>1]}\cosh(\alpha Y_t)\bbE^{\bbQ}[f(\gamma_1)],
\]
which yields the claimed independence.
\end{proof}

Since $\hat{\mu}$ is adapted to $\cF^Y$ by definition, we deduce that the filtered Az\'ema martingale of the first kind is independent of $(\gamma_t)$. This is in stark contrast to Az\'ema's martingale, $\mu$, which is a function of the process  $(\gamma_t)$.

% Recall that the unobserved drift of $Y$ is
% $\int_0^{\cdot}\chf_{[W_{g_s(Y)}>0]}-\chf_{[W_{g_s(Y)}\leq 0]}\,ds$,
% so  it is natural to consider the $\cF^Y_t$-conditional distribution
% of $\int_0^{t}\chf_{[W_{s}>0]}\,ds$.
\section{Filtered Az\'ema martingale of the second kind} \label{s:sfam}
We now return to study the solutions of equation (\ref{e:fam2}) and
the associated projection of $W$. Recall that the   equation (\ref{e:fam2})
is the following SDE:
\be  \label{e:fam21}
Y_t=B_t +\alpha \int_0^t \sgn(W_s)dL_s(Y),
\ee
where $L(Y)$ is the symmetric local time of $Y$ at $0$. The right
local time of $Y$ at $0$ will be denoted with $\ell(Y)$. We will write
$L$ and $\ell$ instead of $L(Y)$ and $\ell(Y)$, respectively, when no confusion arises.

\begin{proposition} \label{p:fam2weak} Suppose that $|\alpha|\leq 1$.
\begin{itemize}
\item[i)] There is a
  unique weak solution to (\ref{e:fam21}). Moreover, $Y\sgn(W_{g(Y)})
  \eid X$, where $X$ is a skew Brownian motion which
  solves (\ref{e:sbm}).
\item[ii)] $|Y|$ is a
  reflecting Brownian motion. The symmetric and
  nonsymmetric local times, $\ell$ and $L$, respectively, of $Y$ at $0$ are related
  by
\[
\ell_t=\int_0^t\left(1+\alpha \sgn(W_s)\right)dL_s.
\]
\end{itemize}
\end{proposition}
\begin{proof} Suppose $X$ is the skew Brownian motion that solves
  (\ref{e:sbm}). As observed in  Introduction, this SDE has a unique
  solution.  Next let $Y_t=\sgn(W_{g_t(X)})X_t$. Observe that $Y$ is
a continuous semimartingale in view of Lemma \ref{l:balayage} and $[X,X]_t=[Y,Y]_t=t$. Moreover,
$L(X)=L(Y)$.  Indeed, (see Exercise VI.1.25 in \cite{RY})
\[
L_t(X)=\lim_{\eps \rar 0}\int_0^t \chf_{[|X_t|< \eps]}(s)ds=\lim_{\eps
  \rar 0}\int_0^t \chf_{[|Y_t|< \eps]}(s)ds=L_t(Y).
\]
Thus, $Y$ satisfies
\bean
Y_t&=&\int_0^t \sgn(W_{g_s(X)})dB_s +\alpha \int_0^t\sgn(W_{g_s(X)})
dL_s(X) \\
&=&\beta_t +\alpha \int_0^t\sgn(W_{s})dL_s(X)\\
&=&\beta_t +\alpha \int_0^t\sgn(W_{s})dL_s(Y),
\eean
where $\beta:=\int_0^{\cdot} \sgn(W_{g_s(X)})dB_s$, the first equality
is due to Lemma \ref{l:balayage} and the second is due to the fact
that support of the measure $dL(X)$ is contained in the zero set of
$X$. This shows that $\sgn(W_{g_t(X)})X_t$ is a weak solution to
(\ref{e:fam21}). By working backwards one can also see that
$\sgn(W_{g(Y)})Y$ is a weak solution to (\ref{e:sbm}). Since  there is
a one-to-one correspondence between $Y$ and $\sgn(W_{g(Y)})Y$, we
obtain the uniqueness in law of the solutions to (\ref{e:fam21}) from
the analogous property of the solutions to (\ref{e:sbm}). Again, since
the solutions to (\ref{e:sbm}) are unique in law, we also have
$\sgn(W_{g(Y)})Y\eid X$. Therefore, $|Y|=|X|$. Since $|X|$ is a
reflecting Brownian motion (see, e.g., Lemma 2.1 in \cite{BBKM}), so is
$|Y|$.

To find the relationship between $\ell$ and $L$, first observe that
\[
\ell_t-\ell^{0-}_t=2\alpha \int_0^t \sgn(W_s)dL_s
\]
by Theorem VI.1.7 in \cite{RY}. Moreover, Exercise VI.1.25 in
\cite{RY} yields
\[
L_t=\frac{\ell_t+\ell^{0-}_t}{2}.
\]
Thus,
\[
\ell_t=\int_0^t\left(1+\alpha\, \sgn(W_s)\right)dL_s.
\]
\end{proof}
The equation (\ref{e:fam21}) in fact has a unique strong solution. We
need the following lemma for the proof.
\begin{lemma} \label{l:slt}
Suppose $X^i=M+A^i$ for $i=1,2$ where $X^i_0=0$,  $M$ is a continuous local
martingale and $A^i$ is continuous and of finite variation for each
$i$.
\begin{itemize}
\item[i)] If $X^i\geq 0$, then $L(X^i)=\int_0^{\cdot}
  \chf_{[X^i_s=0]}dX^i_s$ and $L(X^i)=\half \ell(X^i)$.
\item[ii)]  $2 L(X^{i^+})= L(X^i) +\int_0^{\cdot}
  \chf_{[X^i_s=0]}dX^i_s$.
\item[iii)] $L(X^1 \vee X^2)=\int_0^{\cdot} \chf_{[X^2_s \leq
    0]}dL_s(X^1)+\int_0^{\cdot} \chf_{[X^1_s < 0]}dL_s(X^2)$.
\end{itemize}
\end{lemma}
\begin{proof}
\begin{itemize}
\item[i)] By Tanaka's formula for the symmetric local times (see
  Exercise VI.1.25 in \cite{RY}), we obtain
\be \label{e:nstanaka}
dX^{i^+}_t=\half\left\{2 \chf_{[X^i_t>0]}+\chf_{[X^i_t=0]}\right\}dX^i_t
+\half dL_t(X^i).
\ee
However, since $X^{i^+}=X^i$, we immediately deduce from the above
that
\[
L(X^i)_t=\int_0^{t}
  \chf_{[X^i_s=0]}dX^i_s.
\]
The second assertion follows from Exercise VI.1.16 in \cite{RY}.
\item[ii)] In view of the results from  part i) and (\ref{e:nstanaka})
\bean
dL(X^{i^+})&=&\half \chf_{[X^{i^+}_s=0]}\left(\left\{2 \chf_{[X^i_t>0]}+\chf_{[X^i_t=0]}\right\}dX^i_t
+\half dL_t(X^i)\right)\\
&=&\half \chf_{[X^i_t=0]}dX^i_t+\half dL(X^i)_t
\eean
since $\int_0^t \chf_{[X^i_s\neq 0]}dL_s(X^i)=0$.
\item[iii)] Let $S=X^1\vee X^2$ and observe that since $S=X^1 +
  (X^2-X^1)^+$, by Tanaka formula
\[
dS_t= dM_t + \chf_{[X^2_t>X^1_t]}dA^2_t + \chf_{[X^2_t\leq
  X^1_t]}dA^1_t.
\]
Thus, $S=M+C$ for where $C$ is continuous and of finite variation. By part ii)
\[
L_t(S)=2 L_t(S^+)-\int_0^t \chf_{[S_s=0]}dS_s.
\]
Then, by part i) and Exercise VI.1.21, we obtain
\bean
dL_t(S)&=&\chf_{[X^2_t \leq
    0]}d\ell_t(X^1)+\chf_{[X^1_t < 0]}d\ell_t(X^2)-
  \left(\chf_{[X^1_t=0, X^2_t\leq 0]}+\chf_{[X^2_t=0,X^1_t<0]}\right)dS_t\\
&=&\chf_{[X^2_t \leq
    0]}d\ell_t(X^1)+\chf_{[X^1_t < 0]}d\ell_t(X^2)-
 \left(\chf_{[X^1_t=0, X^2_t\leq
     0]}+\chf_{[X^2_t=0,X^1_t<0]}\right)dC_t\\
&=&\chf_{[X^2_t \leq
    0]}\left\{d\ell_t(X^1)-\chf_{[X_t^1=0]}dA^1_t\right\}+\chf_{[X^1_t<0]}\left\{d\ell_t(X^2)-\chf_{[X^2_t=0]}dA^2_t\right\}\\
&=&\chf_{[X^2_t \leq
    0]}dL_t(X_1)+ \chf_{[X^1_t<0]}dL_t(X^2),
\eean
where the second line is due to Theorem VI.1.7
from \cite{RY}  and the last line follows from the same theorem and  Exercise VI.1.25 in \cite{RY}.
\end{itemize}
\end{proof}
\begin{theorem} \label{t:strongY} Pathwise uniqueness holds for
  (\ref{e:fam21}). Consequently, there is a unique strong
  solution. Moreover, $\sgn(W_{g(Y)})Y$ is a skew Brownian  motion independent of $W$.
\end{theorem}
\begin{proof}
Suppose there are two solutions, $Y^1$ and $Y^2$. Then,
\bean
d(Y^1 \vee Y^2)_t&=& dB_t +\alpha \sgn(W_t)dL_t(Y^1) +
\chf_{[Y^2_t>Y^1_t]}d(Y^2-Y^1)_t\\
&=&dB_t +\alpha \sgn(W_t)dL_t(Y^1) +
\alpha
\chf_{[Y^2_t>Y^1_t]}\sgn(W_t)\left\{dL_t(Y^2)-dL_t(Y^1)\right\}\\
&+&dB_t +\alpha \sgn(W_t)dL_t(Y^1) +
\alpha\chf_{[Y^1_t<0]}\sgn(W_t)dL_t(Y^2)-\alpha\chf_{[Y^2_t>0]}\sgn(W_t)dL_t(Y^1)\\
&=&dB_t +\alpha \chf_{[Y^2_t\leq 0]}\sgn(W_t)dL_t(Y^1) +
\alpha\chf_{[Y^1_t<0]}\sgn(W_t)dL_t(Y^2)\\
&=&dB_t +\alpha\sgn(W_t)dL_t(Y^1 \vee Y^2).
\eean
Thus, $Y^1 \vee Y^2$ is also a solution to (\ref{e:fam21}). However,
since weak uniqueness holds for  (\ref{e:fam21}), we conclude that
$Y^1=Y^2$. Since weak existence and pathwise uniqueness implies the
existence and uniqueness of the strong solutions by the celebrated
Yamada-Watanabe theorem, the second claim follows.

In order to see the claimed independence, let $X=\sgn(W_{g(Y)})Y$. As
observed earlier, due to the balayage formula,
\[
X_t= \beta_t + \alpha L_t(X)
\]
where $\beta$ is a Brownian motion defined by
$\int_0^{\cdot}\sgn(W_{g_s(Y)})\,dB_s$. By Theorem \ref{t:HS}, $X$ is
adapted to the natural filtration of $\beta$. However, $\beta$ is
independent of $W$ since $[W,\beta]=0$.
\end{proof}
The theorem above tells us in particular that the zero set of $Y$ is that of a skew
Brownian motion which is independent of $W$. This will greatly
simplify our computations when we consider the $\cF^Y$-optional
projection of $W$, where $\cF^Y$ is the usual augmentation of the
natural filtration of $Y$ and $Y$ is the unique strong solution of
(\ref{e:fam21}).

For any $t\geq 0$ define the stopping time
\[
d_t(Y)=\inf\{u>t:Y_u=0\}.
\]
Then, we have the following
\begin{proposition} \label{p:sgnisknown} For any $t \geq 0$,
  $\sgn(W_{g_t(Y)})$ is $\cF^Y_t$-measurable. Similarly, $\sgn(W_{d_t(Y)})$ is $\cF^Y_{d_t}$-measurable.
\end{proposition}
\begin{proof} We will first show that $\sgn(W_{g_t(Y)})$ is $\cF^Y_t$-measurable. Since $\ell$ is $\cF^Y$-adapted, we have that
\[
\int_0^t\left(1+\alpha \sgn(W_s)\right)dL_s \in \cF^Y_t
\]
by Proposition \ref{p:fam2weak}. Moreover, since $\ell$ is
$\cF^Y$-optional and $L$ is  $\cF^Y$-adapted and increasing, the
$\cF^Y$-optional projection of $\int_0^{\cdot}\eta_sd\ell_s$, for any
bounded
$\cF^Y$-optional $\eta$, is given by
\[
\int_0^{\cdot}\eta_s\left(1+\alpha {}^o\sgn(W)_s\right)dL_s,
\]
where ${}^o\sgn(W)$ stands for the $\cF^Y$-optional
projection of $\sgn(W)$. Thus, we have
\[
\int_0^{\infty}\eta_s( {}^o\sgn(W)_s-\sgn(W_s))dL_s=0
\]
for any bounded $\cF^Y$-optional $\eta$. Thus, ${}^o\sgn(W)_s=\sgn(W_s)$ if $s$ belongs to the support of $dL$. On the other hand, by Proposition \ref{p:fam2weak}, $|Y|$ is a reflecting Brownian motion. Therefore, the support of $dL$ is `exactly' the zero set of $Y$ (see Proposition VI.2.5 in \cite{RY}). Since $Y_{g_t(Y)}=0$ we deduce that $\sgn(W_{g_t(Y)}) \in \cF_{g_t(Y)}$ since ${}^o\sgn(W)_{g_t(Y)}\in \cF_{g_t(Y)}$. This also implies that
\be \label{e:sgnratio}
\chf_{[Y_t \neq 0]}\sgn(W_{g_{t}(Y)})=\chf_{[Y_t \neq 0]}\frac{Y_t}{X_t}
\ee
where $X$ is a skew Brownian motion adapted to $\cF^Y$ in view of Theorem \ref{t:strongY}.

Next,  consider the sequence of following stopping times:
\[
T_t^n=\inf\{u\geq d_t:|Y_u|=\frac{1}{n}\}.
\]
Clearly, $T^n_t$ is decreasing in $n$ and $\lim_{n \rar \infty}T_t^n=d_t$. Then, by (\ref{e:sgnratio})
\[
\liminf_{n \rar \infty} \sgn(W_{g_{T_t^n}(Y)})=\liminf_{n \rar \infty} \frac{Y_{T^n_t}}{X_{T^n_t}}.
\]
Next, we will show that $\liminf_{n \rar \infty} \sgn(W_{g_{T_t^n}(Y)})=\sgn(W_{d_t}),\, \bbP-$a.s.. To this end, first observe that if $u_n \downarrow u$ then $\sgn(W_{u_n})\rar \sgn(W_u)$ unless $W_u=0$ by the continuity of $u$ and the shape of the $\sgn$ function. Also note that since the mapping $t \mapsto g_t(Y)$ is right continuous, $\lim_{\rar \infty}g_{T_t^n}(Y)=g_{d_t}(Y)=d_t$. However, $d_t$ is independent of $W$ since it is an $\cF^X$-stopping time in view of Theorem \ref{t:strongY}. Thus, $\bbP(W_{d_t}=0)=0$, which in turn yields that
\[
\sgn(W_{d_t})=\liminf_{n \rar \infty} \sgn(W_{g_{T_t^n}(Y)})=\liminf_{n \rar \infty} \frac{Y_{T^n_t}}{X_{T^n_t}} \in \cF^Y_{d_t}
\]
by the right-continuity of the filtration $\cF^Y$ and the fact that $X$ is $\cF^Y$-adapted. Since the filtration is completed by the $\bbP$-null sets, we therefore conclude $\sgn(W_{d_t}) \in \cF^Y_{d_t}$.
\end{proof}
The above result shows that by observing $Y$ we learn the sign of $W$
at the end of every excursion interval of $Y$ (or alternatively of $X$). Let's denote the
 $\cF^Y$-optional projection of $W$ by $\hat{\nu}$. We  call this
 martingale {\em the filtered Az\'ema martingale of the second  kind}.
\begin{corollary} \label{c:nhat} $\hat{\nu}_t=\sgn(W_{g_t(Y)})\sqrt{g_t(Y)}$.
 \end{corollary}
 \begin{proof} Let $X=\sgn(W_{g(Y)})Y$ and recall that  $\cG$ is the usual augmentation of the natural filtration of $\sgn(W)$. Then, in view of Proposition \ref{p:sgnisknown} and  Theorem \ref{t:strongY}, we obtain $\cF^Y_t \subset \cF^X_t \vee \cG_{g_t(Y)}$. To ease the exposition let's denote $g_t(Y)$ with $g_t$. Since $X$ is independent of the filtration $\cG$ and $g_t(Y)=g_t(X)$,
 \be \label{e:ce1}
 \bbE[W_t|\cF^Y_t]=\bbE[\mu_{g_t}|\cF^Y_t]=\sgn(W_{g_t})\sqrt{\frac{\pi}{2}}\bbE\left[\sqrt{g_t-\gamma_{g_t}}\,\big|\cF^Y_t\right],
 \ee
 where $\gamma$ is as in (\ref{e:gamma}). On the other hand, Exercise XII.3.8 in \cite{RY} and the scaling properties of standard Brownian motions together imply that, for any $u$, the process $\left(\frac{W_{s \gamma_u}}{\sqrt{\gamma_u}}\right)_{s \in [0,1]}$ is a Brownian bridge independent of $\gamma_t$. Since $\sgn(W_{s\gamma_u})=\sgn(\frac{W_{s \gamma_u}}{\sqrt{\gamma_u}})$, this yields that $\gamma_u$ is independent of $\sgn(W_r)$ whenever $r \leq \gamma_u$. Moreover, Lemme 1 in \cite{AY} further implies that $\sgn(W_u)$ is independent of $\gamma_u$. Combining these two observations allows us to deduce that  $\gamma_{g_t}$ is independent of $\sigma(\sgn(W_{g_s}), g_s; s \leq t)$ since $(g_s)_{s \geq 0}$ is independent of $W$ by Theorem \ref{t:strongY}. (Recall once again that that $\bbP(\gamma_{g_t}=g_t)=0$ in view of the independence of $W$ and $g$.) Therefore, (\ref{e:ce1}) can be rewritten as
 \[
 \bbE[W_t|\cF^Y_t]=\sgn(W_{g_t})\sqrt{\frac{\pi}{2}}\bbE[\sqrt{g_t-\gamma_{g_t}}]=\sgn(W_{g_t})\sqrt{g_t}
 \]
 since $g_t$ has the arcsine law.
 \end{proof}
The result above means that $\hat{\nu}$ is a pure jump martingale which is constant
 on $[g_t,t]$. Therefore, it is a martingale which can jump only at the
 end of the excursion interval $(g_t(Y),d_t(Y)]$! Also observe that it is equally likely that this martingale will jump or stay constant when the excursion of $Y$ away from $0$ comes to an end.  The presence of a martingale
 with jumps in particular implies that the optional and predictable
 $\sigma$-algebras associated to $\cF^Y$ are different. Recall, however, that
 the martingales adapted to the filtration of the filtered Az\'ema
 martingale of the first kind is continuous implying the equivalence of
 the associated predictable and optional $\sigma$-algebras.

We can also find the $\cF^Y_t$-conditional law of $W_t$ as a straightforward corollary to  Proposition 4 in \cite{AY} and the independence of $\gamma_{g_t(Y)}$ from $\cF^Y_t$ as observed in the proof above.
\begin{corollary} Let $F:\bbR\mapsto\bbR$ be a bounded measurable function. Fix a $t>0$ and define $f:[0,t]\times \bbR\mapsto \bbR$ by $f(s,x)=\int_{\bbR}F(y)p(t-s,y-x)\,dy$ where $p$ is the transition density of standard Brownian motion. Let
\[
h(s,x)=\int_0^{s}f(s,x \sqrt{s-r})\frac{1}{\sqrt{\pi}}\frac{1}{\sqrt{s-r}\sqrt{r}}\,dr.
\]
Then,
\[
\bbE[F(W_t)|\cF^Y_t]=\int_0^{\infty} h\left(g_t(Y), \sgn(W_{g_t(Y)})\frac{\pi}{2}y\right)ye^{-\frac{y^2}{2}}\,dy.
\]
%\[
% g(s,x)=\int_0^{\infty}f(s,x y)ye^{-\frac{y^2}{2}}dy.
%\]
%Then, $\bbE[F(W_t)|\cF^Y_t]=h(g_t(Y),\hat{\nu}_t)$, where $h:[0,t]\times \bbR\mapsto \bbR$ is defined by

\end{corollary}
% \begin{proof} Since $\bbE[F(W_t)|\cF^Y_t]=\bbE[f(g_t(Y), W_{g_t(Y)})|\cF^Y_t]$, the claim directly follows from Proposition 4 in \cite{AY}.
% \end{proof}

\end{document}